\documentclass[11pt]{amsart}

\usepackage[article]{robertpreamble}
\usepackage{geometry}
\begin{document}


\title[Equidistribution of Eigenfunctions of Quantum Cat Maps]{Equidistribution of Eigenfunctions of Quantum Cat Maps}

\author{Robert Koirala}
\address{Department of Mathematics, University of California San Diego}
\email{rkoirala@ucsd.edu}
\date{\today}

\begin{abstract}
    We prove that the short-period eigenfunctions of quantum cat maps constructed by Kim and the author equidistribute on $\bT^2$ in the sense of semiclassical measures. We also show that their logarithmically large $\ell^\infty$-norm is asymptotically concentrated on a bounded number of coordinates. Thus, for this explicit family, strong coordinate localization coexists with semiclassical equidistribution. These results confirm the behavior suggested by earlier numerical evidence of Kim and the author, and contrast with the scarring phenomena for short-period eigenfunctions observed by Faure, Nonnenmacher, and De Bièvre.
\end{abstract}

\maketitle

\section{Introduction}\label{sec:introduction}

Consider a map
\begin{align}\label{eq:def-of-A}
    A=\begin{pmatrix}a&b\\ c&d\end{pmatrix}\in \mathrm{SL}(2,\Z),\qquad
    ab,cd\in 2\Z,\qquad
    2<\tr(A)\in 2\Z,\qquad
    \gcd(b,c)=1.
\end{align}
The map \(A\) induces a \textit{chaotic} automorphism of \(\bT^2\coloneqq \mathbb{R}^{2}/\mathbb{Z}^2\), called a \textit{classical cat map}, in the sense that $\{A^t(x)\colon t\in \Z\}$ equidistributes on $\bT^2$ for almost all $x\in \bT^2$ \cite{MR8647}. For a given $A$ and a Planck parameter $N\in \N$, one associates a unitary operator \(M_{N,\theta}\), called a \textit{quantum cat map}, acting on an \(N\)-dimensional Hilbert space \(\ceH_N(\theta)\) of \textit{quantum states;} see Section \ref{sec:background}. The fundamental link between the classical and quantum dynamics is the exact Egorov identity between any observable $a\in C^\infty(\bT^2)$ and its quantized operator $\Op_{N,\theta}(a)$:
\begin{align}
    M_{N, \theta}^{-1} \Op_{N, \theta} (a) M_{N, \theta} =\Op_{N, \theta} (a \circ A).
\end{align}
Thus the quantum evolution generated by \(M_{N,\theta}\) quantizes the classical evolution generated by \(A\). In this paper we study the semiclassical distribution and coordinate profile, as $N\to \infty$, of a family of short-period eigenfunctions of \(M_{N,0}\) constructed by Kim and the author in \cite{kim-koirala-2023bounds}.

\begin{figure}[!htb]
    \centering
    \includegraphics[width=0.23\linewidth]{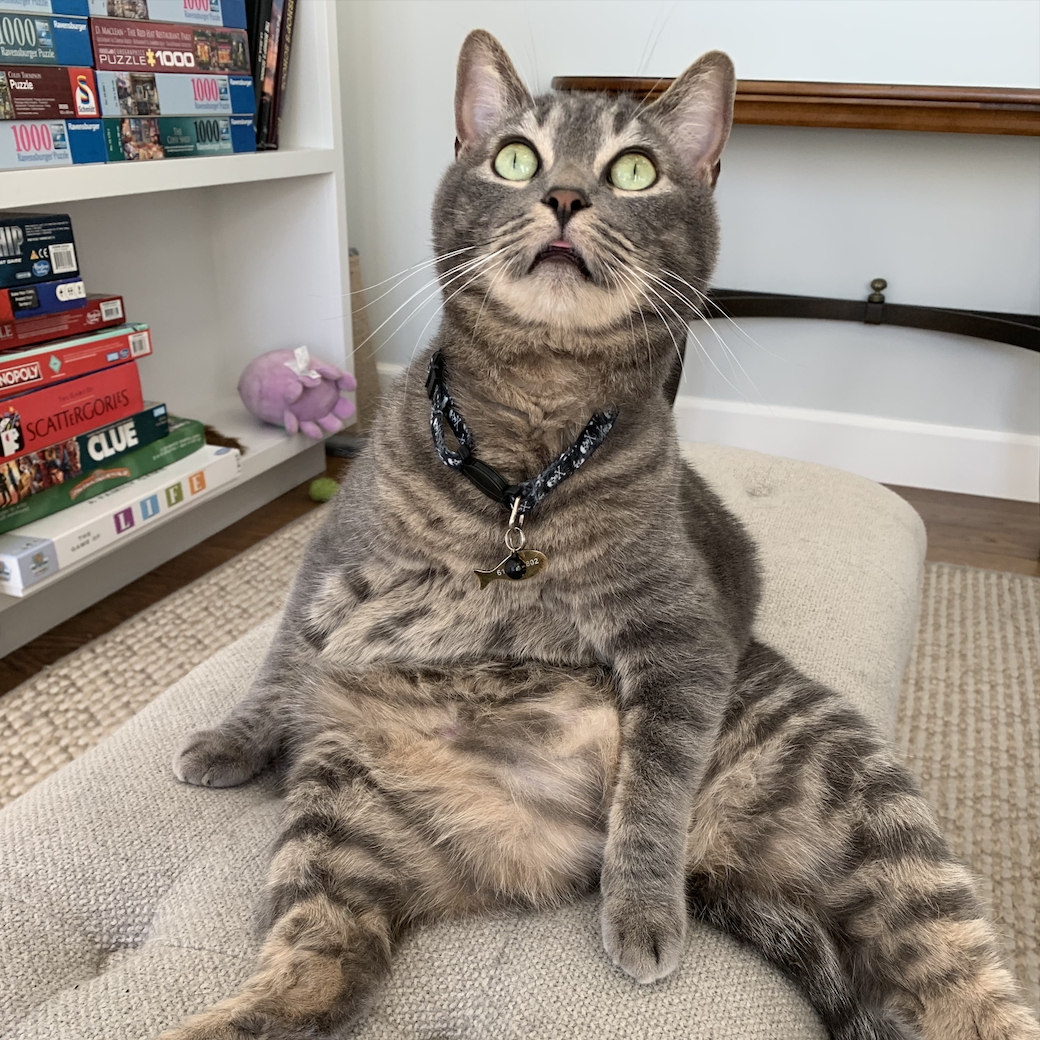}\hfill
    \includegraphics[width=0.23\linewidth]{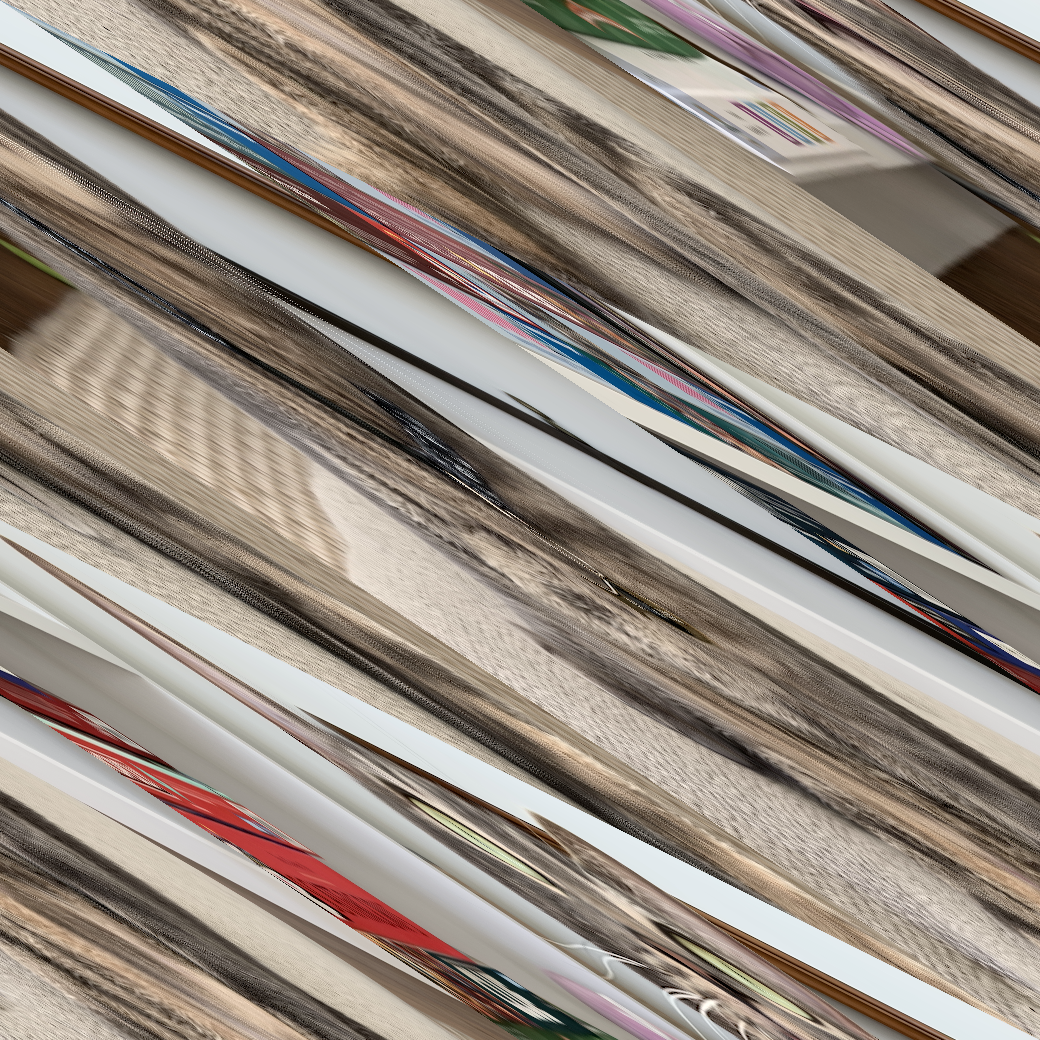}\hfill
    \includegraphics[width=0.23\linewidth]{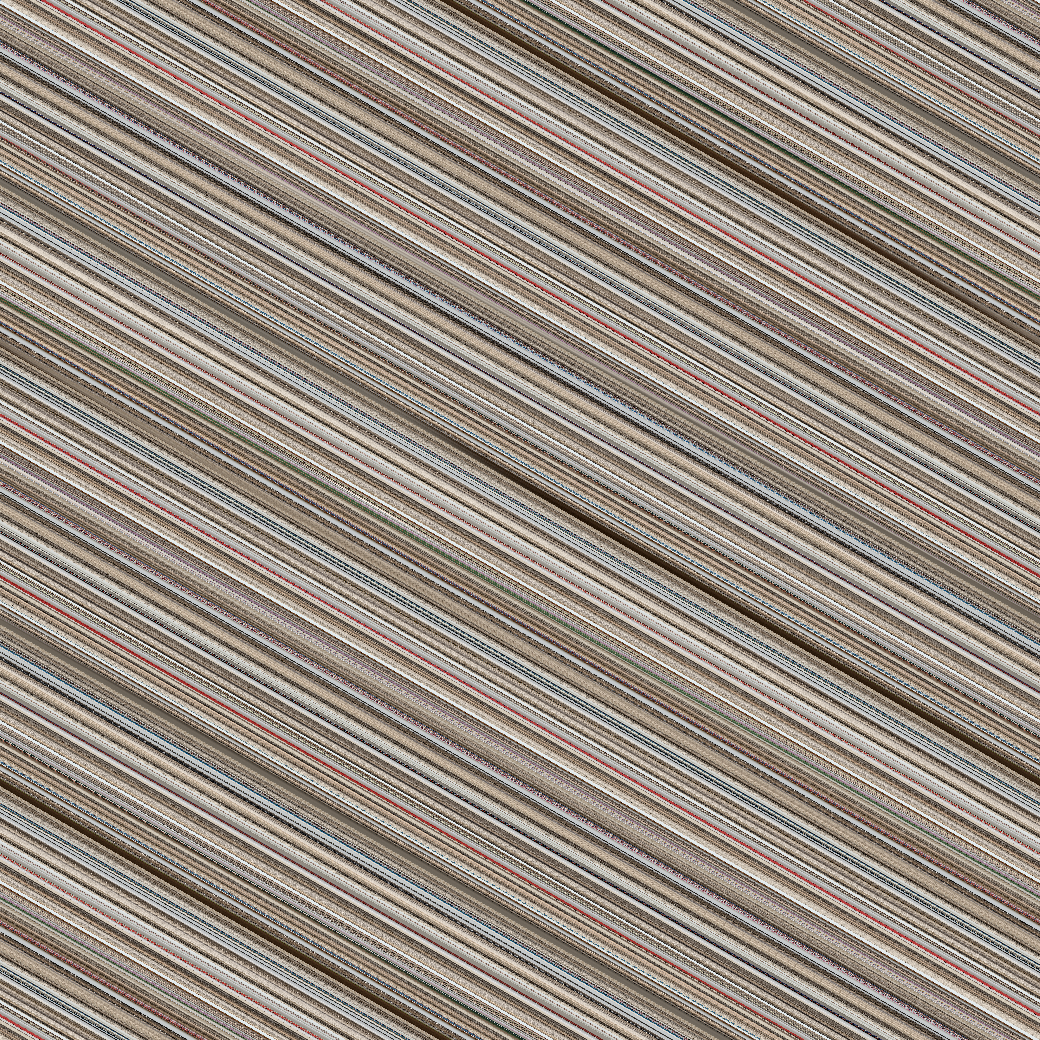}\hfill
    \includegraphics[width=0.23\linewidth]{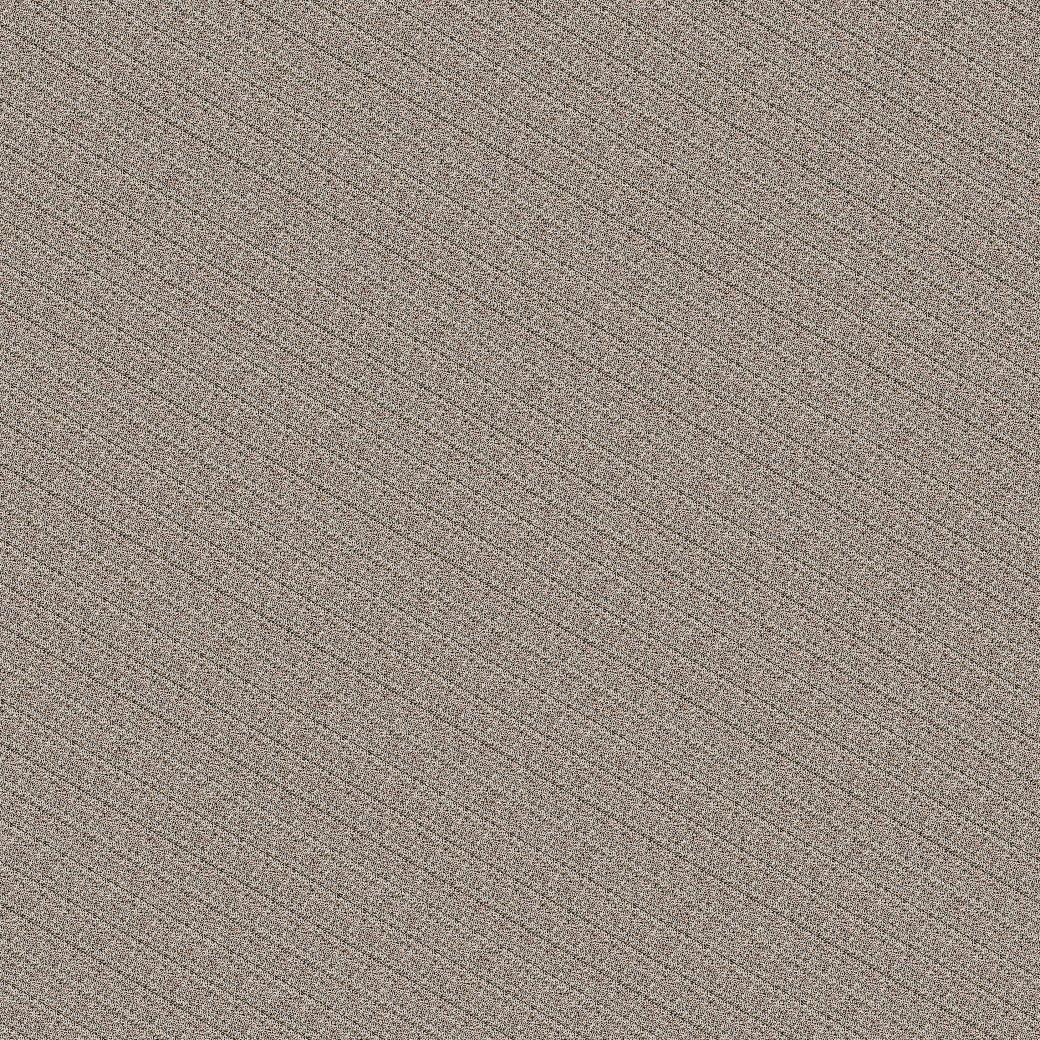}
    \caption{Iterates of the cat map $A=\begin{pmatrix}2&3\\1&2\end{pmatrix}$ applied to an image of a cat. From left to right: the original photograph and its images after $1$, $3$, and $5$ iterations. Original photo courtesy of Laurie Ellsworth.} 
    \label{fig:cat-map}
\end{figure}

Quantum cat maps, introduced by Hannay and Berry \cite{MR602111}, are discrete models in \textit{quantum chaos}, where one studies how chaotic classical dynamics constrain the quantum states in the high-frequency limit. They serve as a testing ground for questions about mass and distribution of the Laplace eigenfunctions on negatively curved manifolds as the eigenvalues go to infinity. A classical cat map plays the role of geodesic flow on negatively curved manifolds which is known to be chaotic \cite{MR224110}. The torus \(\bT^2\) plays the role of position-momentum space, the eigenfunctions of \(M_{N,\theta}\) play the role of Laplace eigenfunctions, and the limit \(N\to\infty\) corresponds to the high-frequency limit.

Bouzouina--De Bi\`evre and Kurlberg--Rudnick proved \textit{quantum ergodicity} for quantum cat maps showing that eigenfunctions of $M_{N,\theta}$ equidistribute in position-momentum space as $N\to \infty$ \cite{BouzouinaDeBievre, MR1853869}. This is analogous to the quantum ergodicity theorem of \cite{ MR818831, MR402834, MR916129} which proves that a density-one subsequence of eigenfunctions equidistributes in the high-frequency limit. Rivi\`ere studied entropy of limits of eigenfunctions of $M_{N,\theta}$ as another measure of quantum chaos \cite{RiviereEntropy}. More recent developments include that of Schwartz on delocalization \cite{SchwartzThesis,SchwartzPAA} and some higher-dimensional works \cite{bhakta2026quantum, MR4703425, MR2630057, kim2024characterizing,MR4927803}. For further related works, see \cite{MR2563805, MR1316757, MR1107009}. 

In the Laplace eigenfunction literature, a stronger \textit{quantum unique ergodicity} (QUE) conjecture of \cite{MR1266075} predicts that, on closed negatively curved manifolds, \textit{every} sequence of eigenfunctions equidistributes toward Liouville measure on the position-momentum space in the high-frequency limit. For related results, see a recent survey \cite{dyatlov2023macroscopic} and the references therein. Unlike the manifold setting, however, the direct analogue of quantum unique ergodicity fails for quantum cat maps. Faure, Nonnenmacher, and De Bi\`evre constructed eigenfunctions of $M_{N,\theta}$ with short period which have nontrivial atomic component and therefore do not converge to the volume measure on $\bT^2$ \cite{MR2036373,MR2000926}. These examples show that short quantum periods can lead to \textit{scarring} phenomena.

\begin{figure}[!htb]
    \centering
    \includegraphics[width=0.35\linewidth]{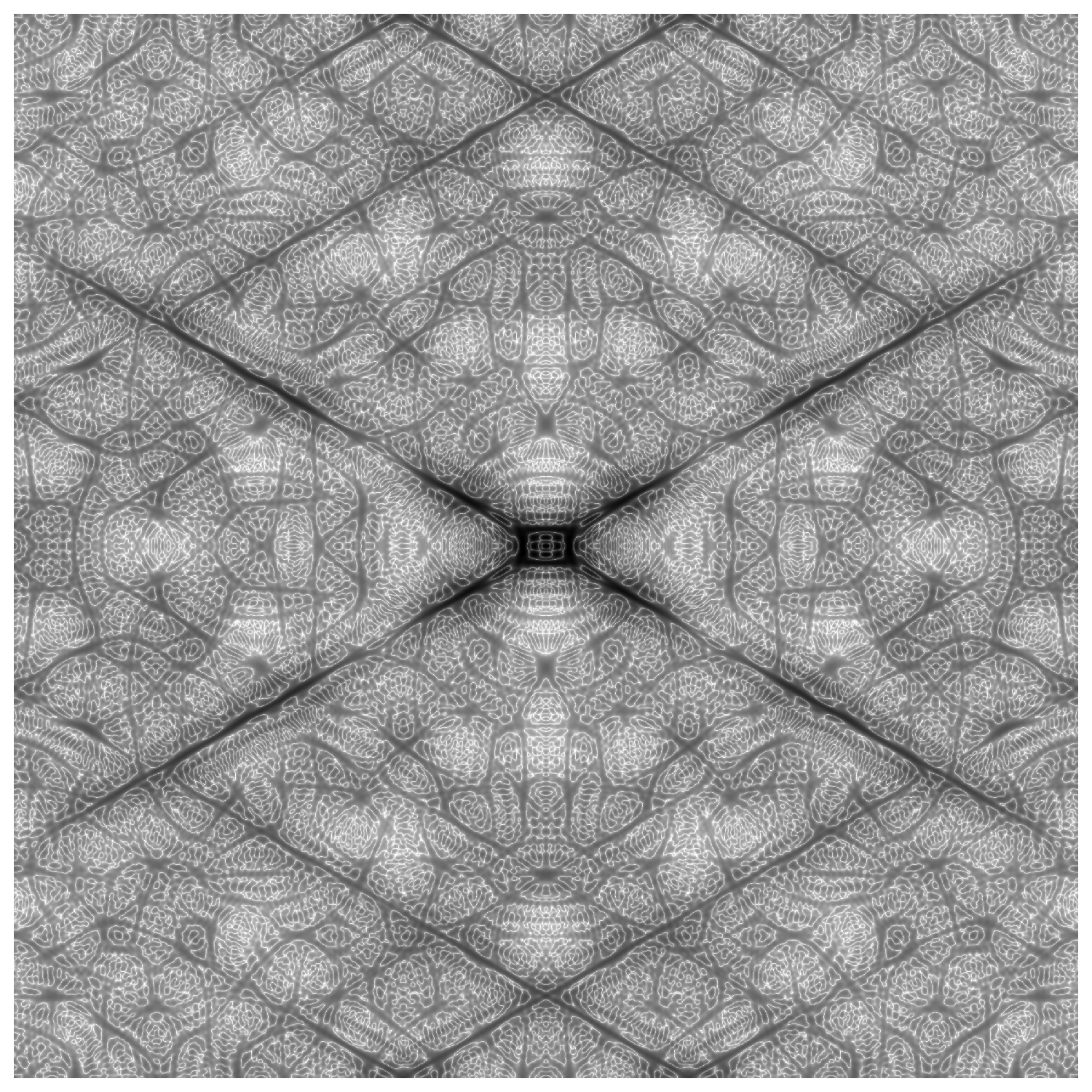} \hspace{0.5in}\includegraphics[width=0.35\linewidth]{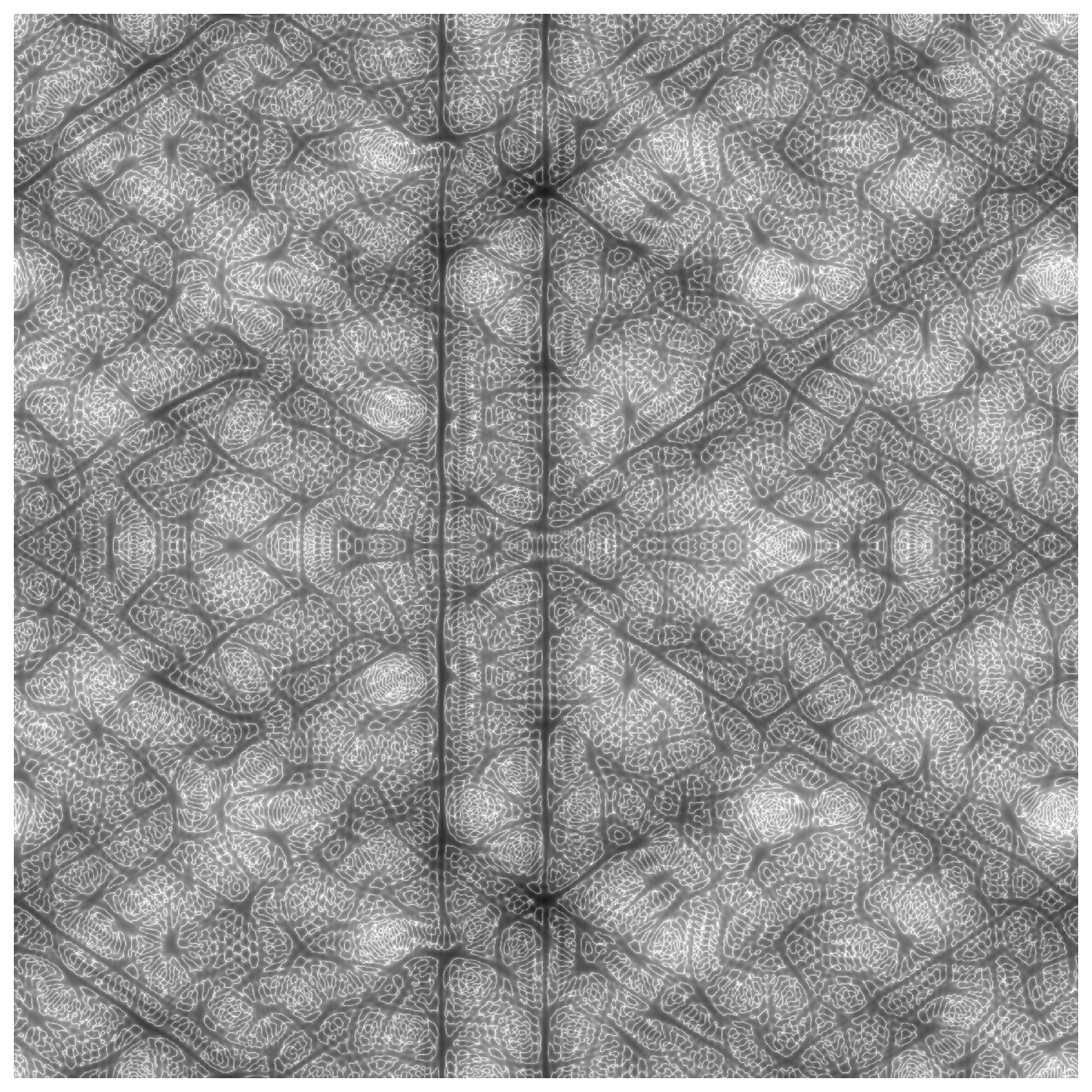}
    \caption{\cite[Figure 4]{kim-koirala-2023bounds}. Gaussian-smoothed Wigner functions of eigenfunctions of the quantized cat map corresponding to $A=\begin{pmatrix}2&3\\1&2\end{pmatrix}$, shown on a logarithmic grayscale. Left: a scarred eigenfunction, with visible concentration along classical periodic structures. Right: an eigenfunction of the form \eqref{eq:vk-intro} whose position-momentum mass is more uniformly distributed, illustrating equidistribution.}
    \label{fig:equidistribution}
\end{figure}

To further understand the short-period behavior, in \cite{kim-koirala-2023bounds}, the authors proved a logarithmic upper bound for the \(\ell^\infty\)-norm of eigenfunctions of quantum cat maps, and showed that this bound is essentially sharp along sequences of large spectral multiplicity by constructing explicit short-period projector states of the form \eqref{eq:vk-intro}. A natural question is whether this explicit near-extremal family behaves like the equidistributed eigenstates appearing in general quantum unique ergodicity results \cite{BouzouinaDeBievre,MR1853869}, or instead resembles the explicit short-period scarred families constructed in \cite{MR2036373,MR2000926}. Numerical evidence in \cite[Figure~4]{kim-koirala-2023bounds} suggested that these near-extremal states might equidistribute in position and momentum. The purpose of the present paper is to prove this phenomenon.

Our main result shows that, for this explicit family, strong localization in the standard coordinate basis coexists with semiclassical equidistribution. More precisely, the normalized short-period projector states have logarithmically large \(\ell^\infty\)-norm, asymptotically concentrated on a bounded number of coordinates, but they converge to Lebesgue measure on \(\bT^2\) as $N\to \infty$.

We now state the results. Fix \(A\) satisfying \eqref{eq:def-of-A}. For \(k\in\N\), let \(N'_k\) be the largest positive integer \(N\) such that \(A^k\equiv I\pmod N\). Given \(N\), let \(n(N)\) denote the quantum period of \(M_{N,0}\):
\begin{align*}
    n(N)\coloneqq
    \min\{t\ge 1: M_{N,0}^t=e^{i\varphi}I
    \text{ for some }\varphi\in\R\}.
\end{align*}
For \(k\in\N_0\), define the odd short-period sequence by
\begin{align}
    N_k\coloneqq N'_{2k+1},\qquad
    t_k\coloneqq n(N_k)=2k+1,\qquad
    M_k\coloneqq M_{N_k,0}.
\end{align}
Then \(M_k^{t_k}=e^{i\varphi_k}I\) for some \(\varphi_k\in\R\). Fix
\(\sigma_k\in\Z\), and set
\begin{align}
    \omega_k\coloneqq
    \exp\!\left(i\frac{\varphi_k+2\pi\sigma_k}{t_k}\right).
\end{align}
For each \(k\), choose \(j_k\in\{0,\dots,N_k-1\}\), and define
\begin{equation}\label{eq:vk-intro}
    v_k\coloneqq
    \frac1{t_k}\sum_{t=0}^{t_k-1}\omega_k^{-t}M_k^t e_{j_k}^0,
    \qquad
    u_k\coloneqq \frac{v_k}{\Verts{v_k}}
\end{equation}
whenever \(v_k\neq0\). Here \(e_j^0\in\ceH_N(0)\) denotes the standard basis vector. Since \(\omega_k^{t_k}=e^{i\varphi_k}\), each nonzero \(u_k\) is an eigenfunction of \(M_k\). Moreover, \(N_k\to\infty\) as \(k\to\infty\) \cite{MR1773813}.

Our first theorem states that, for every observable \(a\in C^\infty(\bT^2)\), the expectation of its quantization \(\Op_{N_k,0}(a)\) in the state \(u_k\) converges to the classical average of \(a\). Put it differently, the eigenfunctions \(u_k\) equidistribute in position-momentum space \(\bT^2\) as $N_k\to \infty$.

\begin{theorem}\label{thm:main-eqd}
    With the notation above, \(v_k\neq 0\) for all sufficiently large \(k\). Moreover, for every \(a\in C^\infty(\bT^2)\),
    \begin{align}\label{eq:main-eqd}
        \angles{\Op_{N_k,0}(a)u_k,u_k}
        =
        \int_{\bT^2} a\,d\mathrm{Leb}
        + O_a\!\left(\frac1{\log N_k}\right)
        \qquad\text{as }k\to\infty.
    \end{align}
\end{theorem}

The notation \(O_a(\cdot)\) means that the implied constant depends on the observable \(a\). 

Although the sequence \(u_k\) equidistributes in position-momentum space, its mass in the standard coordinate basis is concentrated at a single coordinate to leading order.

\begin{figure}[!htb]
    \centering
    \includegraphics[width=1\linewidth]{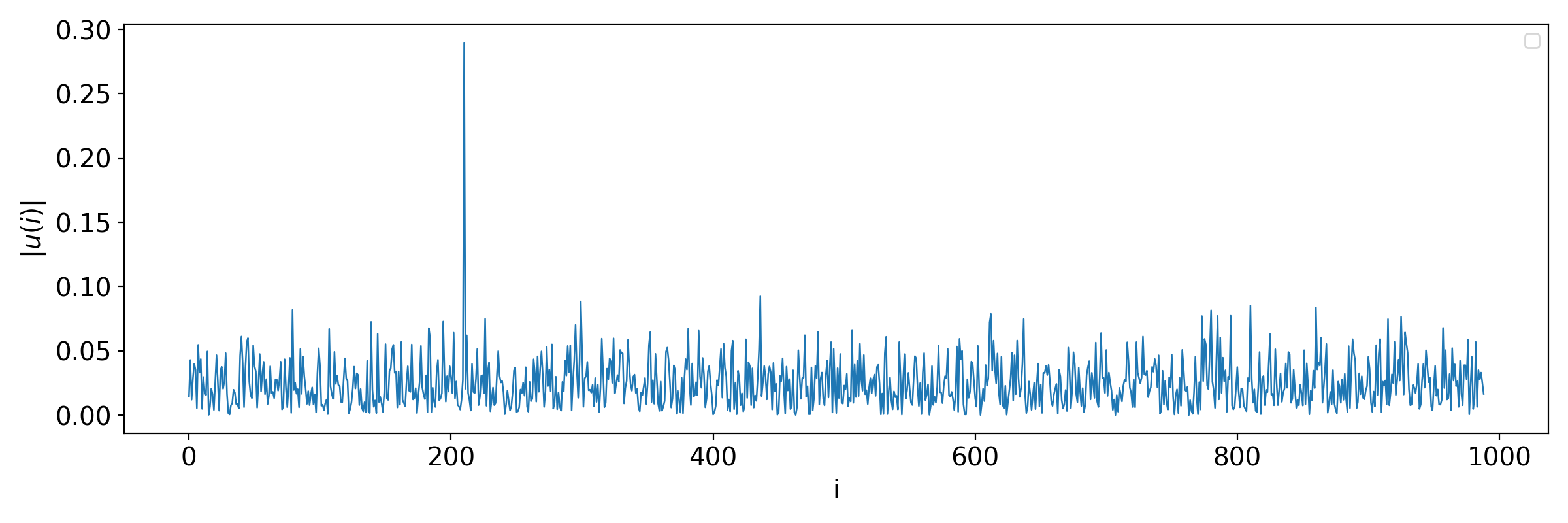}
    \caption{Coordinate profile of an \(\ell^2\)-normalized eigenfunction of \(M_{N,0}\) with large \(\ell^\infty\)-norm, for
    \(A=\begin{pmatrix}2&3\\1&2\end{pmatrix}\) and \(N=989\).}
    \label{fig:coordinate-spike}
\end{figure}

\begin{theorem}\label{thm:main-profile}
    There exists \(c_0>0\), depending only on \(A\), such that the following holds. Write $u_k(\ell)\coloneqq \angles{u_k,e_\ell^0}.$ Then for all sufficiently large \(k\),
    \begin{subequations}
        \begin{align}
            u_k(j_k)&=\frac{1}{\sqrt{t_k}}+O_A\!\bigl(\sqrt{t_k}\,N_k^{-c_0}\bigr),\label{eq:uk-large}\\
            u_k(\ell)&=O_A\!\bigl(\sqrt{t_k}\,N_k^{-c_0}\bigr)
            \qquad (\ell\neq j_k).\label{eq:uk-small}
        \end{align}
    \end{subequations}
    In particular, \(j_k\) is the unique coordinate at which \(|u_k(\ell)|\) is maximized for all sufficiently large \(k\), and
    \begin{align}\label{eq:uk-norm}
        \|u_k\|_{\ell^\infty}
        =
        \frac{1}{\sqrt{t_k}}+O_A\!\bigl(\sqrt{t_k}\,N_k^{-c_0}\bigr).
    \end{align}
\end{theorem}

We also prove analogous results for even short-period sequences; see Section~\ref{sec:even}. In the even case a new phenomenon appears: when the quantum period lies in \(4\N\), the projector state \(v_k\) can vanish. In Proposition~\ref{prop:vanishing-happens}, we construct infinitely many such vanishing examples. Therefore the even analogue of Theorem~\ref{thm:main-eqd} must assume \(v_k\neq0\) along the chosen subsequence. On the other hand, Proposition~\ref{prop:vanishing-is-rare} shows that the vanishing condition is rare as \(k\to\infty\). The coordinate profile is also slightly different: instead of a unique coordinate spike, the mass is concentrated on at most four coordinates.

We briefly describe the proof of Theorem~\ref{thm:main-eqd}. By testing against torus Fourier modes and using Egorov's identity, one reduces the equidistribution statement to estimates for matrix elements of the form
\begin{align}
    \angles{\mathrm{Op}_{N,0}\!\bigl(e^{2\pi i(m_1x+m_2\xi)}\bigr)((A^{\mathsf T})^s m)M^{r-s}e_j^0,e_j^0}, \qquad m=(m_1,m_2)\in \Z^2.
\end{align}
The off-diagonal terms \(r\neq s\) are controlled by a higher-power version of the dispersive Gauss-sum estimate from \cite[Proposition~8]{kim-koirala-2023bounds}, combined with arithmetic estimates for the short-period sequence \(N_k=N'_{2k+1}\). The diagonal terms reduce to counting the times \(0\le s<t_k\) for which the second component of \((A^{\mathsf T})^s m\) vanishes modulo \(N_k\). The key observation is that this component grows like \(\lambda^{|s|}\), while \(N_k\asymp \lambda^{(t_k+1)/2}\). This yields a logarithmic bound for the number of resonant times and hence the rate in \eqref{eq:main-eqd}. The proof of Theorem~\ref{thm:main-profile} uses the same dispersive estimate directly on the standard-basis coordinates of \(v_k\).

\subsection*{Organization}
In Section~\ref{sec:background}, we recall the notation and background on quantum cat maps. In Section~\ref{sec:preliminary-lemma}, we prove the arithmetic and Gauss-sum estimates needed for the main arguments. In Section~\ref{sec:equidistribution}, we prove Theorem~\ref{thm:main-eqd}. In Section~\ref{sec:coordinate}, we prove Theorem~\ref{thm:main-profile}. Finally, in Section~\ref{sec:even}, we treat the even short-period case.

\subsection*{Acknowledgment}
We are thankful to Elena Kim for her detailed comments on earlier drafts and for drawing attention to \cite{SchwartzThesis}.

\section{Background}\label{sec:background}

We recall the notation for quantum cat maps used throughout the paper. We follow \cite[Section~2]{kim-koirala-2023bounds}; see also \cite{MR2000926}.

For a Schwartz symbol \(a\in \mathscr S(\R^2)\) and a semiclassical parameter \(h\in(0,1]\), its Weyl quantization is defined by
\begin{align}\label{eq:weyl-quantization}
    \Op_h(a)f(x)
    \coloneqq
    \frac{1}{2\pi h}
    \int_{\R^2}
    e^{\frac{i}{h}(x-x')\xi}
    a\!\left(\frac{x+x'}2,\xi\right)
    f(x')\,dx'\,d\xi,
    \qquad f\in\mathscr S(\R).
\end{align}
We also use the standard symbol class
\begin{align*}
    S(1)
    =
    \left\{
    a\in C^\infty(\R^2):
    \sup_{(x,\xi)\in\R^2}
    |\partial_{(x,\xi)}^\alpha a(x,\xi)|<\infty
    \text{ for all } \alpha\in\N_0^2
    \right\},
\end{align*}
with seminorms
\begin{align*}
    \|a\|_{C^m}
    \coloneqq
    \max_{|\alpha|\le m}
    \sup_{\R^2}|\partial_{(x,\xi)}^\alpha a|,
    \qquad m\in\N_0.
\end{align*}
For \(a\in S(1)\), the operator \(\Op_h(a)\) acts continuously on \(\mathscr S(\R)\) and by duality on \(\mathscr S'(\R)\).

For \(A\in\mathrm{SL}(2,\R)\), let \(\mathcal M_A\) be the set of unitary operators \(M:L^2(\R)\to L^2(\R)\) satisfying the exact Egorov identity
\begin{equation}\label{eq:MA}
    M^{-1}\Op_h(a)M=\Op_h(a\circ A)
    \qquad\text{for all }a\in S(1).
\end{equation}
Such operators exist and are unique up to multiplication by a complex number of unit modulus.

We now pass from \(\R^2\) to the torus \(\bT^2\coloneqq \R^2/\Z^2\). We identify each \(a\in C^\infty(\bT^2)\) with its \(\Z^2\)-periodic lift to \(\R^2\). Since such a lift belongs to \(S(1)\), the operator \(\Op_h(a)\) is defined on \(L^2(\R)\).

For \(\omega=(y,\eta)\in\R^2\), define the\textit{ quantum translation} $U_\omega$
\begin{align*}
    U_\omega\coloneqq \Op_h(a_\omega),
    \qquad
    a_\omega(z)\coloneqq
    \exp\!\left(\frac{i}{h}\sigma(\omega,z)\right), \qquad \sigma(z,\omega)\coloneqq \xi y-x\eta, \qquad z=(x,\xi)\in \R^2.
\end{align*}
The operators \(U_\omega\) are unitary on \(L^2(\R)\). When \(h=(2\pi N)^{-1},\qquad N\in\N,\) the integer translations \(U_\omega\), \(\omega\in\Z^2\), commute with the quantizations of torus observables:
\begin{equation}\label{eq:commutation_relations}
    \Op_h(a)U_\omega
    =
    U_\omega\Op_h(a),
    \qquad
    a\in C^\infty(\bT^2),\quad \omega\in\Z^2.
\end{equation}
These commutation relations motivate a decomposition of $L^2(\R)$ into a direct integral of finite dimensional spaces $\ceH_N(\theta)$, where $\theta \in \bT^{2}$, such that $\Op_h(a)$ descends onto these spaces. To ensure the these spaces are nontrivial, for the rest of the paper, we assume
\begin{align*}
    h= (2\pi N)^{-1}\quad\text{where }N\in\mathbb N.
\end{align*}
For \(\theta\in\bT^2\), define the space of \textit{quantum states}
\begin{align}
    \ceH_N(\theta)
    \coloneqq
    \left\{
    f\in\mathscr S'(\R):
    U_\omega f
    =
    e^{2\pi i\sigma(\theta,\omega)+N\pi i Q(\omega)}f
    \text{ for all }\omega\in\Z^2
    \right\},
\end{align}
where \(Q(y,\eta)=y\eta\). The following lemma gives an explicit basis for $\ceH_N(\theta)$.

\begin{lemma}[{\cite[Lemma~2.5]{MR4703425}}]\label{lem:basis}
    The space \(\ceH_N(\theta)\) is \(N\)-dimensional. A basis is given by \(\{e_j^\theta\}_{j=0}^{N-1}\), where, for \(\theta=(\theta_x,\theta_\xi)\),
    \begin{align}
        e_j^\theta(x)
        \coloneqq
        \frac{1}{\sqrt N}
        \sum_{k\in\Z}
        e^{-2\pi i\theta_\xi k}
        \delta\!\left(
        x-\frac{Nk+j-\theta_x}{N}
        \right).
    \end{align}
\end{lemma}

We equip \(\ceH_N(\theta)\) with the inner product for which \(\{e_j^\theta\}_{j=0}^{N-1}\) is an orthonormal basis. Although the basis depends on the choice of representative of \(\theta_x\), the inner product depends only on \(\theta\in\bT^2\). If \(u=\sum_{j=0}^{N-1}\alpha_j e_j^\theta,\) we write
\begin{align}
    \|u\|_{\ell^p}
    \coloneqq
    \|(\alpha_0,\dots,\alpha_{N-1})\|_{\ell^p}.
\end{align}

For \(a\in C^\infty(\bT^2)\), define the finite-dimensional quantization
\begin{align*}
    \Op_{N,\theta}(a)
    \coloneqq
    \Op_h(a)|_{\ceH_N(\theta)}
    :
    \ceH_N(\theta)\to\ceH_N(\theta).
\end{align*}
The restriction is well-defined by \eqref{eq:commutation_relations}.

We now specialize to cat maps. Let
\begin{equation}\label{eq:A-matrix}
    A=
    \begin{pmatrix}
        a & b\\
        c & d
    \end{pmatrix}
    \in\mathrm{SL}(2,\Z),
\end{equation}
and choose \(M\in\mathcal M_A\). From the transformation rule for quantum translations,
\begin{align}
    M^{-1}U_\omega M=U_{A^{-1}\omega},
\end{align}
one verifies that
\begin{align}
    M(\ceH_N(\theta))
    \subset
    \ceH_N\!\left(A\theta+\frac{N\phi_A}{2}\right)
\end{align}
for a vector \(\phi_A\in\Z^2\) determined by \(A\). Under the parity assumptions \(ab,cd\in2\Z\) imposed in \eqref{eq:def-of-A}, one has \(\phi_A=0\). Hence \(M\) preserves \(\ceH_N(0)\), and we define
\begin{align}
    M_{N,0}
    \coloneqq
    M|_{\ceH_N(0)}
    :
    \ceH_N(0)\to\ceH_N(0).
\end{align}
The exact Egorov identity then becomes
\begin{align}\label{eq:finite-egorov}
    M_{N,0}^{-1}\Op_{N,0}(a)M_{N,0}
    =
    \Op_{N,0}(a\circ A),
    \qquad a\in C^\infty(\bT^2).
\end{align}

We shall use the following explicit formula for the matrix entries of \(M_{N,0}\), computed in \cite{kim-koirala-2023bounds}:
\begin{align}\label{eq:explicit-formula}
    \angles{M_{N,0}e_j^0,e_k^0}_{\ceH}
    =
    \frac{1}{\sqrt{N|b|}}
    \sum_{r=0}^{|b|-1}
    \exp\!\left[
    \frac{2\pi i}{b}
    \left(
        \frac{ar^2N}{2}
        +arj
        +\frac{aj^2}{2N}
        +\frac{dk^2}{2N}
        -kr
        -\frac{kj}{N}
    \right)
    \right].
\end{align}
When the space is clear from context, we write \(\angles{\cdot,\cdot}_{\ceH}=\angles{\cdot,\cdot}\).

Finally, we record the arithmetic notation associated with the short-period sequences. Fix \(A\) satisfying \eqref{eq:def-of-A}, and let \(\lambda>1\) be its expanding eigenvalue. Define integers \(p_r\) by
\begin{align}
    p_0=0,\qquad
    p_r=\frac{\lambda^r-\lambda^{-r}}{\lambda-\lambda^{-1}}
    \quad (r\ge1).
\end{align}
Then
\begin{align}\label{eq:ar-pr}
    A^r=p_rA-p_{r-1}I,
    \qquad r\ge1.
\end{align}
For \(q\ge1\), let
\begin{align*}
    N'_q
    \coloneqq
    \max\{N\in\N:A^q\equiv I\pmod N\},
\end{align*}
and, for \(N\in\N\), let
\begin{align*}
    T_N\coloneqq
    \min\{t\ge1:A^t\equiv I\pmod N\}.
\end{align*}
By \cite[Proposition~11]{MR1773813},
\begin{align}\label{eq:estimate-on-N}
    N'_{2k+1}=p_k+p_{k+1}\quad (k\ge0),
    \qquad
    N'_{2k}=2p_k\quad (k\ge1),
    \qquad
    T_{N'_q}=q.
\end{align}
The corresponding quantum period is
\begin{align*}
    n(N)
    \coloneqq
    \min\{t\ge1:M_{N,0}^t=e^{i\varphi}I
    \text{ for some }\varphi\in\R\}.
\end{align*}
If \(A_N\) is defined by \(A^{T_N}=I+NA_N,\) then \cite[(36)--(46)]{MR602111} implies that \(n(N)=T_N\) if \(N\) is odd, or if \(N\) is even and \((A_N)_{12}\) and \((A_N)_{21}\) are both even. Otherwise \(n(N)=2T_N\). Consequently,
\begin{align}\label{eq:periods}
    n(N'_{2k+1})=2k+1,
    \qquad
    n(N'_{2k})\in\{2k,4k\}.
\end{align}

\section{Preliminary lemmas}\label{sec:preliminary-lemma}

For \(q\ge 1\), write
\begin{align}
    A^q=
    \begin{pmatrix}
        a_q & b_q\\
        c_q & d_q
    \end{pmatrix}.
\end{align}
Thus \(b_q=bp_q\) and \(c_q=cp_q\), where \(p_q\) is defined in \eqref{eq:ar-pr}. For \(N\in\N\) and \(m=(m_1,m_2)\in\Z^2\), set
\begin{align}
    W_N(m)
    \coloneqq
    \Op_{N,0}\!\bigl(e^{2\pi i(m_1x+m_2\xi)}\bigr).
\end{align}

The following standard facts follow from the explicit formula for Heisenberg translations and from the exact Egorov identity.

\begin{lemma}\label{lem:W-basis}
    For every \(j\in\{0,\dots,N-1\}\),
    \begin{align}
        W_N(m)e_j^0=\gamma_{N,m,j}\,e_{j-m_2}^0
    \end{align}
    for some phase \(\gamma_{N,m,j}\in S^1\), where the index is taken modulo
    \(N\). Moreover, if \(B\coloneqq A^{\mathsf T}\), then
    \begin{align}
        M_{N,0}^{-1}W_N(m)M_{N,0}=W_N(Bm).
    \end{align}
\end{lemma}

The following two lemmata record half period structure of \(M_{N'_{2k},0}\). Since the proof of both lemmata are very similar, we only prove the second lemma.

The next two lemmas record the half-period structure of \(M_{N'_{2k},0}\). The proofs are elementary Gauss-sum computations from \eqref{eq:explicit-formula}. Since the two arguments are very similar, we give the proof of the second lemma.

\begin{lemma}\label{lem:even-half-period-2k}
    For every integer \(k\ge 1\) and every \(j\in\{0,\dots,N'_{2k}-1\}\),
    \begin{align}\label{eq:2k-branch}
        M_{N'_{2k},0}^{k} e_j^0
        =
        \alpha_{k,j}\,e_{a_kj}^0
        +
        \beta_{k,j}\,e_{a_kj+N'_{2k}/2}^0,
    \end{align}
    where all indices are taken modulo \(N'_{2k}\), and \(|\alpha_{k,j}|=|\beta_{k,j}|=\frac1{\sqrt2}.\)
\end{lemma}

\begin{lemma}\label{lem:even-half-period-4k}
    Suppose \(n(N'_{2k})=4k\). Then for every \(j\in\{0,\dots,N'_{2k}-1\}\),
    \begin{align}\label{eq:4k-branch}
        M_{N'_{2k},0}^{2k}e_j^0
        =
        \eta_{k,j}\,e_{j+N'_{2k}/2}^0
    \end{align}
    for some phase \(\eta_{k,j}\in S^1\).
\end{lemma}

\begin{proof}
    For notational simplicity, write \(N_k\coloneqq N'_{2k}=2p_k,\) and \(M_k\coloneqq M_{N_k,0}.\) 

    We first note that \(b\) and \(c\) are odd, while \(a\) and \(d\) are even. Indeed, if \(b\) were even, then \(\gcd(b,c)=1\) would force \(c\) odd, and \(cd\in 2\mathbb Z\) would imply \(d\) even. Since \(\tr(A)\) is even, \(a\) would also be even, and then \(\det(A)=ad-bc\) would be even, a contradiction. Thus \(b,c\) are odd and \(a,d\) are even.
    \begin{claim}\label{claim:oddity}
        The integers 
        \(R\coloneqq |b|\frac{\tr(A^k)}{2}\) and $a_{2k}$ are odd.
    \end{claim}

    \begin{proof}[Proof of Claim \ref{claim:oddity}]
        By the discussion preceding \eqref{eq:periods}, \(N_k=2p_k\) and the fact that \(b,c\) are odd, it follows that
        \begin{align}
            m_k\coloneqq \frac{p_k\tr(A^k)}{2p_k} =\frac{p_{2k}}{N_k}
        \end{align}
        is odd. Also, \(a_{2k}=ap_{2k}-p_{2k-1}=am_kN_k-p_{2k-1}.\) Because \(a\) is even and \(p_{2k-1}\) is odd, it follows that \(a_{2k}\) is odd.
    \end{proof}

    \begin{claim}\label{claim:at-most-two-support}
        The support of \(M_k^{2k}e_j^0\) is contained in the coordinates corresponding to \(\ell_0\coloneqq a_{2k}j\) and \(\ell_1\coloneqq a_{2k}j+\frac{N_k}{2}.\)
    \end{claim}

    \begin{proof}[Proof of Claim \ref{claim:at-most-two-support}]
        After absorbing irrelevant phases, we may use \eqref{eq:explicit-formula} to obtain
        \begin{align}
            \langle M_k^{2k}e_j^0,e_\ell^0\rangle =
            \frac{e^{i\theta_{j,\ell}}}{\sqrt{N_k|b_{2k}|}}
            \sum_{q=0}^{|b_{2k}|-1}
            \exp\!\left(
            \frac{2\pi i}{b_{2k}}
            \Bigl(
            \frac{a_{2k}N_k}{2}q^2-(\ell-a_{2k}j)q
            \Bigr)\right).
        \end{align} 
        Let \(\varepsilon\coloneqq \operatorname{sgn}(b),\) so that \(b_{2k}=\varepsilon RN_k\). Substituting
        \begin{align*}
            q=u+2Rv, \qquad u\in\{0,\dots,2R-1\}, \qquad v\in\{0,\dots,N_k/2-1\}.
        \end{align*}
        gives
        \begin{align*}
            \langle M_k^{2k}e_j^0,e_\ell^0\rangle
            =
            \frac{e^{i\theta'_{j,\ell}}}{N_k\sqrt R}
            \sum_{u=0}^{2R-1}
            \exp\!\left(
            2\pi i\Bigl(
            \varepsilon\frac{a_{2k}u^2}{2R}
            -\varepsilon\frac{(\ell-a_{2k}j)u}{RN_k}
            \Bigr)\right)
            \sum_{v=0}^{N_k/2-1}
            \exp\!\left(
            -2\pi i\,2\varepsilon\frac{\ell-a_{2k}j}{N_k}v
            \right).
        \end{align*}
        The inner geometric sum vanishes unless \(\ell-a_{2k}j\equiv 0\pmod{N_k/2},\) which implies the claim.
    \end{proof}

    \begin{claim}\label{claim:l-vanishes}
        We have \(\angles{M_k^{2k}e_j^0, e_{\ell_0}^0}=0\).
    \end{claim}

    \begin{proof}[Proof of Claim \ref{claim:l-vanishes}]
        Since \(A^{2k}\equiv I\pmod{N_k}\), we have \(a_{2k}\equiv 1\pmod{N_k}\), so
        \begin{align*}
                \ell_0\equiv j\pmod{N_k},
            \qquad
            \ell_1\equiv j+\frac{N_k}{2}\pmod{N_k}.
        \end{align*} 
        Now set \(\ell=\ell_\delta\coloneqq a_{2k}j+\delta N_k/2\), \(\delta\in\{0,1\}\). Then
        \begin{align*}
            \langle M_k^{2k}e_j^0,e_{\ell_\delta}^0\rangle
            =
            \frac{e^{i\vartheta_{j,\delta}}}{2\sqrt R}
            \sum_{u=0}^{2R-1}
            \exp\!\left(
            2\pi i\Bigl(
            \varepsilon\frac{a_{2k}u^2}{2R}
            -\varepsilon\frac{\delta u}{2R}
            \Bigr)\right).
        \end{align*}
        Define
        \begin{align*}
            F_\delta(u)\coloneqq 
            \exp\!\left(
            2\pi i\Bigl(
            \varepsilon\frac{a_{2k}u^2}{2R}
            -\varepsilon\frac{\delta u}{2R}
            \Bigr)\right).
        \end{align*}
        Pair the terms \(u\) and \(u+R\). Since \(a_{2k}\) and \(R\) are odd,
        \begin{align*}
            \frac{F_\delta(u+R)}{F_\delta(u)}
            =
            \exp\!\left(
            2\pi i\varepsilon
            \Bigl(
            a_{2k}u+\frac{a_{2k}R}{2}-\frac{\delta}{2}
            \Bigr)\right)
            =
            (-1)^{1-\delta}.
        \end{align*}
        Hence \(F_0(u+R)=-F_0(u)\) and \(F_1(u+R)=F_1(u).\) Therefore the sum for \(\delta=0\) vanishes which implies the claim.
    \end{proof}
    Thus, \(M_k^{2k}e_j^0\) is supported on at most one basis vector, namely \(e_{j+N_k/2}^0\). Since \(M_k^{2k}\) is unitary,  there exists a phase \(\eta_{k,j}\in S^1\) such that \eqref{eq:4k-branch} holds.
\end{proof}

The next lemma is the higher-power version of the dispersive bound proved in \cite[Proposition~8]{kim-koirala-2023bounds}. Related matrix-entry calculations for short quantum periods also appear in \cite[(4.8) Section~4.1]{SchwartzThesis}. The proof is the same Gauss-sum computation as in \cite[Proposition~8]{kim-koirala-2023bounds}, applied to \(A^r\) in place of \(A\), so we omit it.

\begin{lemma}\label{lem:gauss}
    Let $N\in \mathbb N$ and $r\in \mathbb Z\setminus\{0\}$. Then for every $m\in \mathbb Z^2$ and every $j,\ell\in\{0,\dots,N-1\}$,
    \begin{align}
        \bigl|\angles{W_N(m)M_{N,0}^{\,r}e_j^0, e_\ell^0}\bigr| \le \frac{\sqrt{\gcd(N,b_{|r|})}}{\sqrt N}.
    \end{align}
\end{lemma}

The next lemma is a slightly sharper form of a divisibility estimate related to \cite[Lemma~4.1.3]{SchwartzThesis}, where the author proves $\gcd(N'_T,b_r)= O(\sqrt{N'_{T}})$ via a different method.

\begin{lemma}\label{lem:gcd}
    For any integer $T\ge 1$ and $1\leq r<T$,
    \begin{align}
        \gcd(N_T',b_r)\le |b|\,N'_{\gcd(T,2r)}.
    \end{align}
\end{lemma}

\begin{proof}
    Since $b_r=bp_r$, it suffices to prove that
    \begin{align}
        \gcd(N'_T,p_r)\mid N'_{\gcd(T,2r)}.\label{eq:gcd-lucas}
    \end{align}
    
    Since \(\gcd(N'_T,p_r)\mid N'_T\), we have
    \begin{align}
        A^T\equiv I\pmod{\gcd(N'_T,p_r)}.
    \end{align}
    Since \(\gcd(N'_T,p_r)\mid p_r\), the identity $A^r=p_rA-p_{r-1}I$ implies $A^r\equiv -p_{r-1}I\pmod{\gcd(N'_T,p_r)}.$ Taking determinants yields $p_{r-1}^2\equiv 1\pmod{\gcd(N'_T,p_r)},$ and hence
    \begin{align}
        A^{2r}\equiv I\pmod{\gcd(N'_T,p_r)}.
    \end{align}
    Therefore the order of \(A\) modulo \(\gcd(N'_T,p_r)\) divides both \(T\) and \(2r\), so it divides $\gcd(T,2r)$. Thus
    \begin{align}
        A^{\gcd(T,2r)}\equiv I\pmod{\gcd(N'_T,p_r)}.
    \end{align} 
    By definition, $A^{\gcd(T,2r)}\equiv I\pmod{N'_{\gcd(T,2r)}}.$ Hence
    \begin{align}
        A^{\gcd(T,2r)}\equiv I \pmod{\operatorname{lcm}(\gcd(N'_T,p_r),\,N'_{\gcd(T,2r)})}.
    \end{align}
    Since \(N'_{\gcd(T,2r)}\) is the largest positive integer with this property, it follows that
    \begin{align}
        \operatorname{lcm}(\gcd(N'_T,p_r),\,N'_{\gcd(T,2r)})=N'_{\gcd(T,2r)},
    \end{align}
    and therefore $\gcd(N'_T,p_r)\mid N'_{\gcd(T,2r)}$ which is \eqref{eq:gcd-lucas}.
\end{proof}

We conclude the preparatory section with a quantitative estimate on the number of times a nonzero Fourier mode can have vanishing second component modulo \(N\) along the orbit of \(A^{\mathsf T}\).

\begin{lemma}\label{lem:resonance}
    There exists a constant \(C_A>0\), depending only on \(A\), such that for every \(m\in\Z^2\setminus\{0\}\), every \(c\in\Z/N\Z\), and all sufficiently large integers \(T\), with \(N\coloneqq N'_T\),
    \begin{align}
        \#\Bigl\{
        0\le s<T:
        e_2\cdot (A^{\mathsf T})^s m\equiv c\pmod N
        \Bigr\}
        \le
        C_A\bigl(1+\log(1+\|m\|)\bigr).
        \label{eq:resonance-bound}
    \end{align}
\end{lemma}

\begin{proof}
    Set \(B\coloneqq A^{\mathsf T}\) and
    \begin{align*}
        w_s\coloneqq e_2\cdot B^s m,\qquad s\in\Z.
    \end{align*}
    The sequence \(w_s\) is not identically zero. Indeed, if \(w_s=0\) for all \(s\), then \(w_0=m_2=0\), while
    \begin{align*}
        w_1=e_2\cdot Bm=bm_1+dm_2=bm_1.
    \end{align*}
    Since \(b\neq0\), this would force \(m_1=0\), contradicting \(m\neq0\).

    \begin{claim}\label{claim-ws-estimate}
        There exists a constant $C_1>0$ such that
        \begin{align}
            |w_s| \le C_1\|m\|\lambda^{|s|}
            \qquad\text{for all } s\in\Z.
        \end{align}
    \end{claim}

    \begin{proof}[Proof of Claim \ref{claim-ws-estimate}]
        Since \(B\) is hyperbolic with eigenvalues \(\lambda\) and \(\lambda^{-1}\), the sequence \((w_s)_{s\in\Z}\) satisfies the second-order recurrence
        \begin{align}
            w_{s+2}-(\tr A)w_{s+1}+w_s=0.
        \end{align}
        Hence there exist \(\alpha,\beta\in\R\), not both zero, such that
        \begin{align}
            w_s=\alpha\lambda^s+\beta\lambda^{-s}
            \qquad\text{for all } s\in\Z.
        \end{align}
        Moreover, \(\alpha\) and \(\beta\) depend linearly on \(m\). Indeed,
        \begin{align}
            w_0=m_2,\qquad w_1=e_2\cdot Bm=bm_1+dm_2,
        \end{align}
        and solving $\alpha+\beta=w_0,$ $\alpha\lambda+\beta\lambda^{-1}=w_1$ gives
        \begin{align*}
            \alpha=\frac{w_1-\lambda^{-1}w_0}{\lambda-\lambda^{-1}},
            \qquad
            \beta=\frac{\lambda w_0-w_1}{\lambda-\lambda^{-1}}.
        \end{align*}
        Therefore there exists a constant \(C_1=C_1(A)>0\) such that
        \begin{align}
            |\alpha|+|\beta|\le C_1\|m\|
        \end{align}
        from which the claim follows.
    \end{proof}

    \begin{claim}\label{claim:ws-vanishes}
        There exists \(L_m\le C_2\bigl(1+\log(1+\|m\|)\bigr)\) such that the following holds. If \(|s|\le T/2-L_m\) and \(w_s\equiv c\pmod N\), then \(w_s=\widetilde{c}\), where \(\tilde c\in[-N/2,N/2)\) is a representative of $c$.
    \end{claim}

    \begin{proof}[Proof of Claim \ref{claim:ws-vanishes}]
        By \eqref{eq:ar-pr} and \eqref{eq:estimate-on-N}, there exists \(c_A>0\) such that
        \begin{align}
            N'_T\ge c_A\lambda^{(T+1)/2}
        \end{align}
        for all sufficiently large \(T\). Choose \(L_m\ge1\) such that
        \begin{align*}
            C_1\|m\|\lambda^{-L_m}<\frac{c_A\lambda^{1/2}}{2}.
        \end{align*}
        Then \(L_m\le C_2\bigl(1+\log(1+\|m\|)\bigr)\) for some \(C_2=C_2(A)>0\). If \(|s|\le T/2-L_m\), then
        \begin{align*}
            |w_s|
            \le C_1\|m\|\lambda^{T/2-L_m}
            <\frac12 c_A\lambda^{(T+1)/2}
            \le \frac{N'_T}{2}.
        \end{align*}
        Therefore, for \(|s|\le T/2-L_m\), the congruence $w_s\equiv c\pmod N$ is equivalent to the exact equality \(w_s=\tilde c.\)
    \end{proof}

    We now count solutions in three ranges. First suppose
    \(0\le s\le T/2-L_m\). Then \(w_s=\tilde c\). Since
    \begin{align*}
        \alpha\lambda^s+\beta\lambda^{-s}=\tilde c
    \end{align*}
    is equivalent, after multiplying by \(\lambda^s\), to
    \begin{align*}
        \alpha(\lambda^s)^2-\tilde c\,\lambda^s+\beta=0,
    \end{align*}
    it has at most two real solutions in \(s\), and hence at most two integer solutions.

    Next suppose \(T/2+L_m\le s<T\). Write \(s=T-r\), so \(0\le r\le T/2-L_m\). Since \(A^T\equiv I\pmod N\), we have \(B^T\equiv I\pmod N\). Hence
    \begin{align*}
        e_2\cdot B^s m
        =
        e_2\cdot B^{T-r}m
        \equiv
        e_2\cdot B^{-r}m
        =
        w_{-r}
        \pmod N.
    \end{align*}
    As above, because \(|-r|\le T/2-L_m\), the congruence \(w_{-r}\equiv c\pmod N\) is equivalent to \(w_{-r}=\tilde c\). The equation
    \begin{align*}
        \alpha\lambda^{-r}+\beta\lambda^r=\tilde c
    \end{align*}
    again has at most two real solutions in \(r\).

    Finally, the middle window \(\left|s-\frac{T}{2}\right|<L_m\) contains \(O(L_m)\) integers. Combining the three ranges gives
    \begin{align*}
        \#\Bigl\{
        0\le s<T:
        e_2\cdot B^s m\equiv c\pmod N
        \Bigr\}
        \le
        C_A\bigl(1+\log(1+\|m\|)\bigr),
    \end{align*}
    after increasing \(C_A\). This proves the lemma.
\end{proof}

\section{Equidistribution of the odd short-period family}
\label{sec:equidistribution}

\begin{proof}[Proof of Theorem~\ref{thm:main-eqd}]
    Fix \(k\), and abbreviate
    \begin{align*}
        N\coloneqq N_k,\qquad
        t\coloneqq t_k,\qquad
        M\coloneqq M_k,\qquad
        j\coloneqq j_k,\qquad
        \omega\coloneqq \omega_k.
    \end{align*}
    For \(m\in\Z^2\), expanding the definition of \(v_k\) gives
    \begin{equation}\label{eq:Wvv-expand-final}
        \angles{W_N(m)v_k,v_k}
        =
        \frac1{t^2}
        \sum_{r,s=0}^{t-1}
        \omega^{-(r-s)}
        \angles{W_N(m)M^r e_j^0,M^s e_j^0}.
    \end{equation}
    We first prove the normalization estimate. Taking \(m=0\) in \eqref{eq:Wvv-expand-final}, the diagonal terms \(r=s\) contribute exactly \(1/t\). For \(r\neq s\), write \(q\coloneqq r-s\). By Lemma~\ref{lem:gauss},
    \begin{align*}
        \bigl|\angles{M^q e_j^0,e_j^0}\bigr|
        \le
        \frac{\sqrt{\gcd(N,b_{|q|})}}{\sqrt N}.
    \end{align*}
    Since \(t=2k+1\) is odd, every proper divisor of \(t\) is at most \(t/3\). Moreover, because \(t\) is odd, \(\gcd(t,2|q|)=\gcd(t,|q|).\) Therefore Lemma~\ref{lem:gcd} gives, for \(1\le |q|\le t-1\),
    \begin{align*}
        \gcd(N,b_{|q|})
        \le
        |b|\,N'_{\gcd(t,|q|)}
        \ll_A \lambda^{t/6}.
    \end{align*}
    On the other hand, \(N=N'_t\asymp_A \lambda^{(t+1)/2}.\) Hence there exists \(c_0>0\), depending only on \(A\), such that
    \begin{equation}\label{eq:offdiag-basic}
        \bigl|\angles{M^q e_j^0,e_j^0}\bigr|
        \ll_A N^{-c_0}
        \qquad (1\le |q|\le t-1).
    \end{equation}
    Summing the off-diagonal terms in \eqref{eq:Wvv-expand-final} and using \(t\asymp_A \log N\), we obtain
    \begin{equation}\label{eq:vk-norm-final}
        \|v_k\|^2
        =
        \frac1t+O_A(N^{-c_0})
        =
        \frac1t(1+o(1)).
    \end{equation}
    In particular, \(v_k\neq0\) for all sufficiently large \(k\).

    We now estimate the nonzero Fourier modes. Fix \(m\in\Z^2\setminus\{0\}\), and set \(B\coloneqq A^{\mathsf T}\). By Lemma~\ref{lem:W-basis} and the exact Egorov identity,
    \begin{align}\label{eq:egorov-matrix-element}
        \angles{W_N(m)M^r e_j^0,M^s e_j^0}
        =
        \angles{W_N(B^s m)M^{r-s}e_j^0,e_j^0}.
    \end{align}
    If \(r\neq s\), then Lemma~\ref{lem:gauss} and the estimate above give
    \begin{align*}
        \bigl|
        \angles{W_N(B^s m)M^{r-s}e_j^0,e_j^0}
        \bigr|
        \ll_A N^{-c_0},
    \end{align*}
    uniformly in \(r,s\). Hence the total off-diagonal contribution to \(\angles{W_N(m)v_k,v_k}\) is \(O_A(N^{-c_0})\). Dividing by \(\|v_k\|^2\) and using \eqref{eq:vk-norm-final}, the corresponding contribution to \(\angles{W_N(m)u_k,u_k}\) is \(O_A(tN^{-c_0}).\)
    
    It remains to estimate the diagonal terms \(r=s\). By Lemma~\ref{lem:W-basis}, \(\angles{W_N(B^s m)e_j^0,e_j^0}\) vanishes unless the second component of \(B^s m\) vanishes modulo \(N\). Therefore
    \begin{align*}
        \left|
        \frac1{t^2}
        \sum_{s=0}^{t-1}
        \angles{W_N(B^s m)e_j^0,e_j^0}
        \right|
        \le
        \frac1{t^2}
        \#\Bigl\{
        0\le s<t:
        e_2\cdot B^s m\equiv 0\pmod N
        \Bigr\}.
    \end{align*}
    By Lemma~\ref{lem:resonance},
    \begin{align}
        \#\Bigl\{
        0\le s<t:
        e_2\cdot B^s m\equiv 0\pmod N
        \Bigr\}
        \ll_A
        1+\log(1+\|m\|).
    \end{align}
    After dividing by \(\|v_k\|^2\), the diagonal contribution to \(\angles{W_N(m)u_k,u_k}\) is therefore \(O_A\!\left( \frac{1+\log(1+\|m\|)}{t} \right).\)

    Combining the diagonal and off-diagonal estimates, we obtain, for every \(m\in\Z^2\setminus\{0\}\),
    \begin{equation}\label{eq:modewise-final}
        \angles{W_N(m)u_k,u_k}
        =
        O_A(tN^{-c_0})
        +
        O_A\!\left(
        \frac{1+\log(1+\|m\|)}{t}
        \right).
    \end{equation}
    This proves \eqref{eq:main-eqd} when $a=e^{2\pi im\cdot z}$. The general case $a\in C^\infty(\bT^2)$ follows by density of the functions $e^{2\pi im\cdot z}$ in $C^\infty (\bT^2)$.
\end{proof}

\section{Coordinate profile}\label{sec:coordinate}

\begin{proof}[Proof of Theorem~\ref{thm:main-profile}]
    Fix \(k\), and abbreviate
    \begin{align}
        N\coloneqq N_k,\qquad t\coloneqq t_k,\qquad M\coloneqq M_k,\qquad j\coloneqq j_k,\qquad \omega\coloneqq \omega_k.
    \end{align}
    As in the proof of Theorem~\ref{thm:main-eqd}, there exists \(c_0>0\), depending only on \(A\), such that
    \begin{equation}\label{eq:odd-coordinate-dispersive-final}
        \bigl|\angles{M^r e_\alpha^0,e_\beta^0}\bigr| \ll_A N^{-c_0} \qquad \text{for all }\alpha,\beta\in\{0,\dots,N-1\} \text{ and all }1\le r\le t-1.
    \end{equation}
    Now expand the coordinates of \(v_k\):
    \begin{equation}\label{eq:vk-coordinate-expand-final}
         v_k(\ell)
         =
         \frac1t\sum_{r=0}^{t-1}\omega^{-r}\angles{M^r e_j^0,e_\ell^0}.
    \end{equation}
    We first consider the distinguished coordinate \(\ell=j\). Since \(\angles{M^0 e_j^0,e_j^0}=1\), while the terms with \(1\le r\le t-1\) satisfy
    \eqref{eq:odd-coordinate-dispersive-final}, we get
    \begin{align}
        v_k(j)
        =
        \frac1t+\frac1t\sum_{r=1}^{t-1}\omega^{-r}\angles{M^r e_j^0,e_j^0}
        =
        \frac1t+O_A(N^{-c_0}).\label{eq:vk-main-coordinate-final}
    \end{align}
    Now let \(\ell\neq j\). Then the \(r=0\) term in \eqref{eq:vk-coordinate-expand-final} vanishes, so
    \begin{align}
         v_k(\ell)
         =
         \frac1t\sum_{r=1}^{t-1}\omega^{-r}\angles{M^r e_j^0,e_\ell^0}
         =
         O_A(N^{-c_0}),\label{eq:vk-off-coordinate-final}
    \end{align}
    uniformly in \(\ell\neq j\). Next, from the proof of Theorem~\ref{thm:main-eqd},
    \begin{equation}\label{eq:vk-inverse-norm-sharp-final}
        \Verts{v_k}^{-1}
        =
        \sqrt t\,\bigl(1+O_A(tN^{-c_0})\bigr).
    \end{equation}
    Multiplying \eqref{eq:vk-main-coordinate-final} by \eqref{eq:vk-inverse-norm-sharp-final}, we get
    \begin{align}
        u_k(j)=v_k(j)\,\Verts{v_k}^{-1} =
        \left(\frac1t+O_A(N^{-c_0})\right)
        \left(\sqrt t\,\bigl(1+O_A(tN^{-c_0})\bigr)\right) =
        \frac1{\sqrt t}+O_A(\sqrt t\,N^{-c_0}),\label{eq:uk-main-coordinate-final}
    \end{align}
    which proves \eqref{eq:uk-large}. Likewise, multiplying \eqref{eq:vk-off-coordinate-final} by \eqref{eq:vk-inverse-norm-sharp-final}, we obtain
    \begin{align}
        u_k(\ell)
        =
        O_A(N^{-c_0})\cdot \sqrt t\,\bigl(1+O_A(tN^{-c_0})\bigr)
        =
        O_A(\sqrt t\,N^{-c_0})
        \qquad (\ell\neq j),\label{eq:uk-off-coordinate-final}
    \end{align}
    which proves \eqref{eq:uk-small}. Since \(tN^{-c_0}\to 0\), for all sufficiently large \(k\) we have
    \begin{align}
        |u_k(j_k)|\ge \frac{3}{4\sqrt{t_k}},
        \qquad
        |u_k(\ell)|\le \frac{1}{4\sqrt{t_k}}
        \quad (\ell\neq j_k).
    \end{align}
    Therefore \(j_k\) is the unique maximizing coordinate of \(u_k\) for all sufficiently large \(k\). Moreover, \eqref{eq:uk-norm} follows from \eqref{eq:uk-main-coordinate-final}.
\end{proof}

\section{Even short-period eigenfunctions}\label{sec:even}

\begin{theorem}\label{thm:even-main}
    Consider the even short-period sequence \(N_k=N'_{2k}\). If \(t_k=2k\) for all sufficiently large \(k\), then \(v_k\neq0\) for all sufficiently large \(k\), and \eqref{eq:main-eqd} holds. If \(t_k=4k\) along an infinite subsequence, then \eqref{eq:main-eqd} holds along every further subsequence for which \(v_k\neq0\). In either case, there exists a set of indices \(S_k\subset\{0,\dots,N_k-1\},\qquad |S_k|\le4,\) such that
    \begin{subequations}
        \begin{align}
            u_k(\ell)
            &=
            O_A\!\bigl(\sqrt{k}\,N_k^{-c_0}\bigr)
            \qquad &&(\ell\notin S_k),\\
            |u_k(\ell)|
            &\asymp_A k^{-1/2}
            \qquad &&(\ell\in S_k).
        \end{align}
    \end{subequations}
\end{theorem}

\begin{figure}[!htb]
    \centering
    \includegraphics[width=\linewidth]{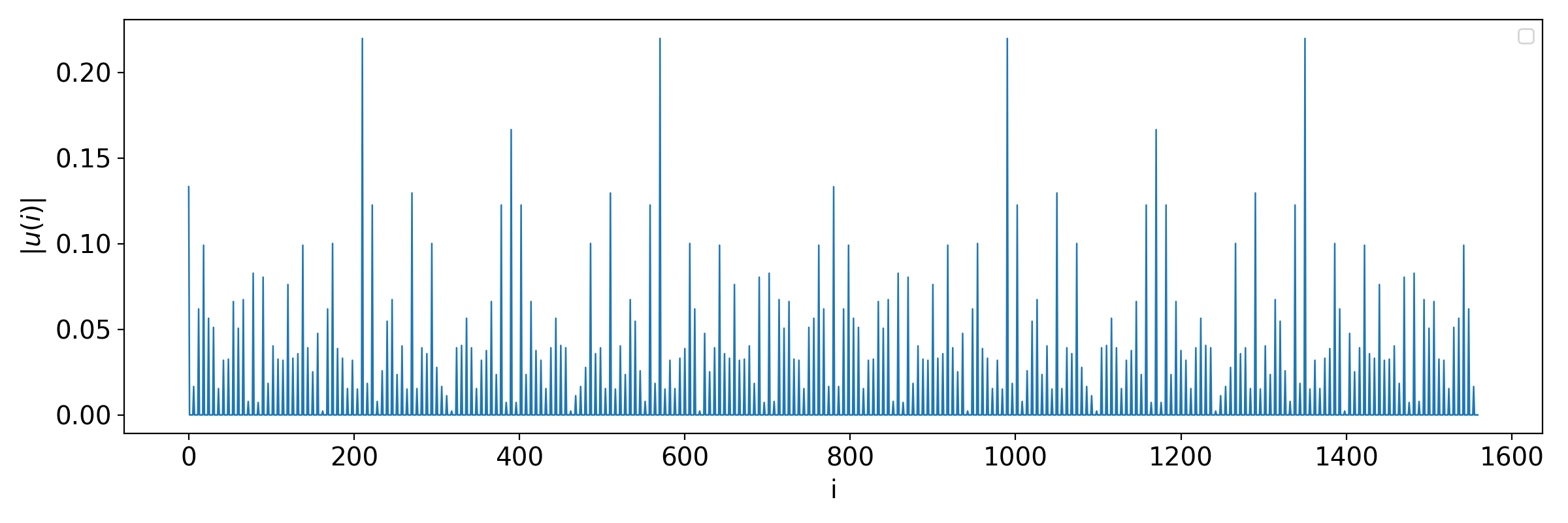}
    \includegraphics[width=\linewidth]{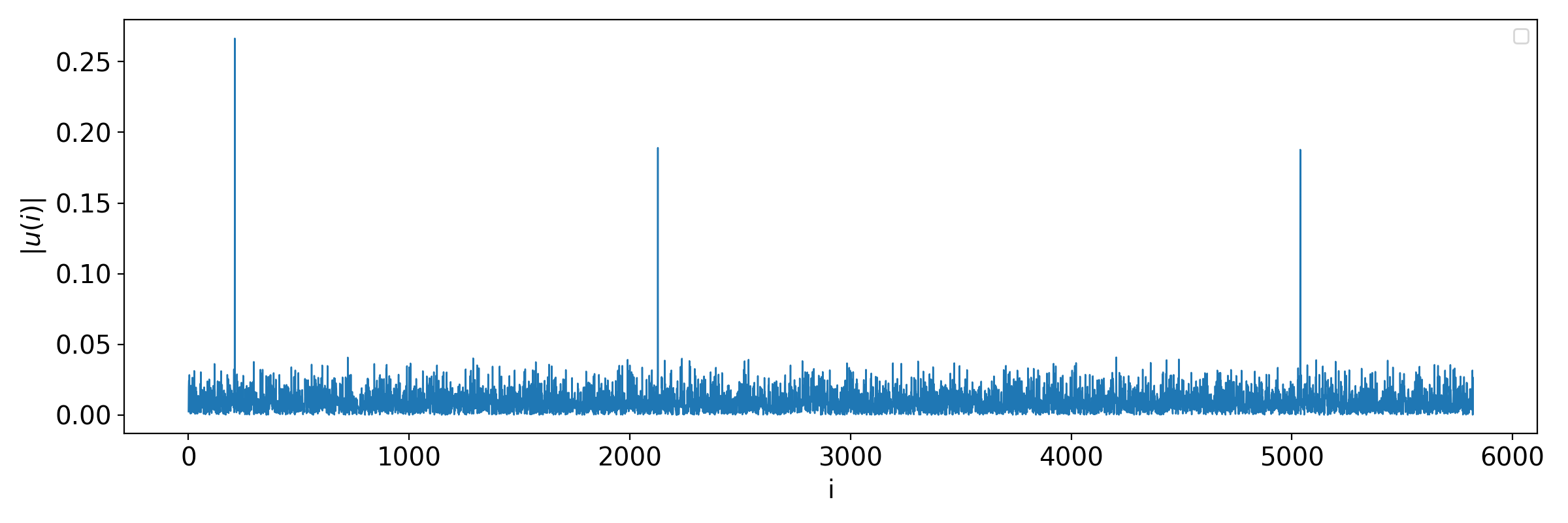}
    \caption{The plots of a maximal $\ell^\infty$-norm, $\ell^2$-normalized eigenfunction of $M_{N,0}$, where $M_{N,0}$ corresponds to $A=\begin{pmatrix}
         2 & 3\\1 & 2
    \end{pmatrix}$. Top: $N=1560$. Bottom: $N=5822$.}
    \label{fig:placeholder}
\end{figure}

\begin{proof}
    Fix \(k\), and abbreviate \(N\coloneqq N_k,\) \(p=p_k\) \(M\coloneqq M_k,\) \(j\coloneqq j_k,\) and \(\omega\coloneqq \omega_k.\)

    First let $t_k=2k$. Set \(g_k\coloneqq (I+\omega^{-k}M^k)e_j^0\) so that 
    \begin{align}
        v_k
        =
        \frac1{2k}\sum_{r=0}^{2k-1}\omega^{-r}M^r e_j^0
        =
        \frac1{2k}\sum_{r=0}^{k-1}\omega^{-r}M^r g_k.
    \end{align}
    By Lemma~\ref{lem:even-half-period-2k}, the vector \(M^k e_j^0\) is supported on two basis vectors with coefficients of modulus \(1/\sqrt2\). Hence \(g_k\) is supported on at most three basis vectors, and
    \begin{align*}
        \|g_k\|^2
        =
        2+2\Re\bigl(\omega^{-k}\langle M^k e_j^0,e_j^0\rangle\bigr)
        \ge 2-\sqrt2>0.
    \end{align*}
    Using
    \begin{align*}
        \|v_k\|^2
        =
        \frac{1}{4k^2}\sum_{r,s=0}^{k-1}\omega^{-(r-s)}
        \langle M^{r-s}g_k,g_k\rangle
    \end{align*}
    and analyzing the diagonal \(r=s\) and off-diagonal \(r\neq s\) as in the proof of Theorem \ref{thm:main-eqd}, we obtain the desired equidistribution result. The result on support and mass follow from an argument similar to that in the proof of Theorem \ref{thm:main-profile}.

    Now let \(t_k=4k\). Set
    \begin{align*}
        Q\coloneqq \omega^{-1}M,\qquad T\coloneqq Q^k,\qquad S\coloneqq Q^{2k}.
    \end{align*}
    Then \(Q^{4k}=I\), \(T^2=S\), \(S^2=I\), and \(S\) commutes with \(T\). By Lemma~\ref{lem:even-half-period-4k}, \(Se_j^0=\xi_{k,j}e_{j+p}^0\) for some \(\xi_{k,j}\in S^1\). Hence
    \begin{align*}
        g_k\coloneqq (I+S)e_j^0=e_j^0+\xi_{k,j}e_{j+p}^0,
        \qquad
        \|g_k\|^2=2.
    \end{align*} 
    Define \(h_k\coloneqq (I+T)g_k.\) Grouping the \(4k\)-term average into residue classes modulo \(k\), we get
    \begin{align*}
        v_k
        =\frac{1}{4k}\sum_{r=0}^{k-1}Q^r h_k.
    \end{align*}
    
    \begin{claim}\label{claim:h-does-not-vanish}
        For sufficiently large \(k\), \(\Verts{h_k}\geq 2\) and \(\# \supp (h_k) \leq 4\).
    \end{claim}

    \begin{proof}[Proof of Claim \ref{claim:h-does-not-vanish}]
        For \(m\in\mathbb Z/p\mathbb Z\), define the two-dimensional block \(W_m\coloneqq \operatorname{span}\{e_m^0,e_{m+p}^0\}.\) The operator \(S\) preserves each \(W_m\), and its \(+1\)-eigenspace in \(W_m\) is one-dimensional. Let \(L_m^+\subset W_m\) denote this line. 
        
        Since \(g_k\in L_j^+\), and by Lemma~\ref{lem:even-half-period-2k} we have \(T(W_m)=W_{a_km},\) \(m\in\mathbb Z/p\mathbb Z,\) while \(T\) commutes with \(S\), we have \(T(L_m^+)=L_{a_km}^+.\) Therefore \(h_k=(I+T)g_k\) is supported on at most the two blocks \(W_j\cup W_{a_kj}.\) In particular, \(\#\supp(h_k)\le4.\)

        To establish the lower bound, choose unit vectors \(f_m^+\in L_m^+\). Since \(g_k=\sqrt2\,f_j^+\) up to phase, there exists \(\rho_{k,j}\in S^1\) such that
        \begin{align*}
            h_k=\sqrt2\bigl(f_j^+ + \rho_{k,j}f_{a_kj}^+\bigr).
        \end{align*}
        If \(a_kj\not\equiv j\pmod p\), the two lines lie in orthogonal blocks, and \(\|h_k\|^2=4.\) If instead \(a_kj\equiv j\pmod p\), then \(T\) preserves \(L_j^+\). Since \(T^2=S\) and \(S=I\) on \(L_j^+\), the action of \(T\) on \(L_j^+\) is multiplication by \(\pm1\). Hence either \(h_k=0\) or \(\|h_k\|^2=8\). Since \(v_k\neq0\), we must have \(h_k\neq0\), and therefore \(\|h_k\|\ge2.\)
    \end{proof}
    As in the case \(t_k=2k\), the Theorem follows from arguments similar to that in the proof of Theorem \ref{eq:main-eqd} and Theorem \ref{thm:main-profile}.
\end{proof}

The following result constructs vanishing projector states $v_k=0$ for infinitely many $k$ when $t_k=4k$. Thus one can't have $v_k\neq 0$ for sufficiently large $k$ unlike in the case of $t_k=2k$ or $t_k=2k+1$. Hence the non-vanishing assumption in Theorem \ref{thm:even-main} is required when $t_k=4k$.
\begin{proposition}\label{prop:vanishing-happens}
    Assume that \(t_k=4k\) for infinitely many \(k\). Then there exist \(\sigma\in\{0,2\}\) and an infinite set \(\mathcal K\subset \N\) such that
    \begin{align}
        v_k=\frac1{4k}\sum_{t=0}^{4k-1}\omega_k^{-t}M_k^t e_0^0=0 \qquad\text{for every }k\in\mathcal K,
    \end{align}
    where \(\omega_k\coloneqq \exp\!\Bigl(i\frac{\varphi_k+2\pi\sigma}{4k}\Bigr)\) and \(j_k\coloneqq 0\).
\end{proposition}

\begin{proof}
    Let \(\mathcal K_0\subset\mathbb N\) be an infinite set on which \(t_k=4k\). Fix \(k\in\mathcal K_0\), and take \(j=0\). For a branch \(\sigma\in\mathbb Z\), set
    \begin{align*}
        \omega_{k,\sigma}\coloneqq \exp\!\Bigl(i\frac{\varphi_k+2\pi\sigma}{4k}\Bigr),
        \qquad
        Q_\sigma\coloneqq \omega_{k,\sigma}^{-1}M_k,
    \end{align*}
    and define \(T_\sigma\coloneqq Q_\sigma^k,\) and \(S_\sigma\coloneqq Q_\sigma^{2k}.\) As in the proof of Theorem \ref{thm:even-main}, set
    \begin{align*}
        v_{k,\sigma}
        =
        \frac1{4k}\sum_{r=0}^{k-1}Q_\sigma^r h_\sigma,
        \qquad
        h_\sigma\coloneqq (I+T_\sigma)(I+S_\sigma)e_0^0.
    \end{align*}
    Since \(0\) is fixed by the block map \(m\mapsto a_km\pmod{p_k}\), the line \(L_0^+\subset W_0\) is preserved by \(T_\sigma\). Also \((I+S_\sigma)e_0^0\in L_0^+\) is nonzero. Hence \(T_\sigma\) acts on this line by a sign:
    \begin{align*}
        T_\sigma (I+S_\sigma)e_0^0
        =
        \varepsilon_{k,\sigma}(I+S_\sigma)e_0^0,
        \qquad
        \varepsilon_{k,\sigma}\in\{\pm1\}.
    \end{align*}
    Therefore \(h_\sigma = (1+\varepsilon_{k,\sigma})(I+S_\sigma)e_0^0.\) 
    
    Now compare the two branches \(\sigma=0\) and \(\sigma=2\). Since \(\omega_{k,2}^{-k}=-\omega_{k,0}^{-k},\) and \(\omega_{k,2}^{-2k}=\omega_{k,0}^{-2k},\) we have \(T_2=-T_0\) and \(S_2=S_0.\) Thus \(\varepsilon_{k,2}=-\varepsilon_{k,0}.\) Consequently exactly one of \(h_0\) and \(h_2\) is zero. When \(h_\sigma\) vanishes so does  \(v_{k,\sigma}.\) For each \(k\in\mathcal K_0\), choose \(\sigma_k\in\{0,2\}\) such that \(v_{k,\sigma_k}=0\). Since \(\mathcal K_0\) is infinite, one of the two choices \(\sigma=0\) or \(\sigma=2\) occurs for infinitely many \(k\). Passing to that infinite subset \(\mathcal K\), we obtain a fixed \(\sigma\in\{0,2\}\) such that \(v_k=0\) for all \(k\in\mathcal K.\)
\end{proof}

Nevertheless, the following proposition shows that the density of indices $j$ at which 
\begin{align}
    v_k=\frac1{4k}\sum_{t=0}^{4k-1}\omega_k^{-t}M_k^t e_j^0=0
\end{align}
for large $k$ is small as $k$ goes to $\infty$.
\begin{proposition}\label{prop:vanishing-is-rare}
    Let \(t_k=4k\). Then for sufficiently large \(k\),
    \begin{align}
        \frac{\# \{j\in\{0,\dots,N_k-1\}: v_k=0\}}{N_k}\ll_{A} N_k^{-1/2}.
    \end{align}
\end{proposition}

\begin{proof}
    By the proof of Claim \ref{claim:h-does-not-vanish}, we know that $v_k$ can vanish only when \(a_kj\equiv j\pmod{p_k}.\) The number of such residue classes modulo \(p_k\) is \(\gcd(a_k-1,p_k)\). Since $N'_k\ll_A \lambda^{k/2}$ and $N_k\asymp_A \lambda^k$, our result follows from the following claim.
    \begin{claim}
        \(\gcd(a_k-1,p_k)=N_k'.\)
    \end{claim}
    \begin{proof}
        First, we show that \(\gcd(a_k-1,p_k)\le N'_k\). Since \(\gcd(a_k-1,p_k)\mid p_k\), we have
        \begin{align*}
            \gcd(a_k-1,p_k)\mid b_k=bp_k, \qquad \gcd(a_k-1,p_k)\mid c_k=cp_k,
        \end{align*}
        and by definition \(\gcd(a_k-1,p_k)\mid (a_k-1).\) Because \(\det(A^k)=1\), it follows that \(d_k\equiv 1\pmod{\gcd(a_k-1,p_k)}\).
        Hence
        \begin{align*}
            A^k\equiv I\pmod{\gcd(a_k-1,p_k)},
        \end{align*}
        so by maximality of \(N'_k\), we obtain \(\gcd(a_k-1,p_k)\le N'_k.\)
        
        Conversely, since \(A^k\equiv I\pmod{N'_k}\), we have
        \begin{align*}
            N'_k\mid (a_k-1),
            \qquad
            N'_k\mid b_k=bp_k,
            \qquad
            N'_k\mid c_k=cp_k.
        \end{align*}
        Since \(\gcd(b,c)=1\), the last two divisibilities imply \(N'_k\mid p_k\). Therefore \(N'_k\) divides both \(a_k-1\) and \(p_k\), and hence \(N'_k\mid \gcd(a_k-1,p_k).\)
    \end{proof}
\end{proof}


\bibliographystyle{plain}
\bibliography{references}

\end{document}